\documentclass[12pt]{amsart}
\usepackage{amssymb,enumerate}
\usepackage{tikz}
\usepackage{amsmath,amssymb,amscd,amsthm,amsxtra}
\usepackage{mathtools}
\usepackage{latexsym,cite}
\mathtoolsset{showonlyrefs}
\def\XXint#1#2#3{{\setbox0=\hbox{$#1{#2#3}{\int}$} 
\vcenter{\hbox{$#2#3$}}\kern-.5\wd0}}

 \newcommand{\bZ}{\mathbb{Z}}
\newcommand{\bR}{\mathbb{R}} \newcommand{\bN}{\mathbb{N}}

\newcommand{\vf}{\mathbf{f}}

\newcommand{\ve}{\mathbf{e}}

\newcommand{\bT}{\mathbb{T}}
\newcommand{\bC}{\mathbb{C}}
 
\newcommand{\cS}{\mathcal{S}}

\def\G{\mathcal{G}}

\DeclareMathOperator*{\esssup}{ess\,sup}
\newcommand{\ra}{\rightarrow}

\newcommand{\bey}{\begin{eqnarray*}}
\newcommand{\eey}{\end{eqnarray*}}
\newcommand{\ba}{\begin{align}}
\newcommand{\ea}{\end{align}}
\newcommand{\bea}{\begin{align*}}
\newcommand{\ena}{\end{align*}}
\newcommand{\be}{\begin{equation}}
\newcommand{\ee}{\end{equation}}
\newcommand{\R}{\mathbb R}
\newcommand{\Z}{\mathbb Z}

\newcommand{\C}{\mathbb C}
\newcommand{\N}{\mathbb N}

\newcommand{\A}{\mathcal A}

\newcommand{\T}{\mathbb T }

\newcommand{\cT}{\mathcal T }

\newcommand{\ZZ}{\mathcal Z}

\newcommand{\bc}{\begin{center}}
\newcommand{\ec}{\end{center}}

\newcommand{\vt}{\boldsymbol\tau}

\DeclareMathOperator{\Span}{Span}

\newtheorem{theorem}{Theorem}[section]

\newtheorem{corollary}[theorem]{Corollary}
\newtheorem{prop}[theorem]{Proposition}

\theoremstyle{definition}

\newtheorem{defn}[theorem]{Definition}
\theoremstyle{remark}
\newtheorem{remark}[theorem]{Remark}

\begin{document}




\title[]{On Schauder bases properties of multiply generated Gabor Systems }
\author{Morten Nielsen}

\begin{abstract} Let $\A$ be a finite subset of $L^2(\R)$ and $p,q\in\bN$. We characterize the Schauder basis properties in $L^2(\bR)$ of the Gabor system
$$G(1,p/q,\A)=\{e^{2\pi i m x}g(x-np/q) : m,n\in \Z, g\in \A\},$$
 with a specific ordering on $\Z\times \Z\times\A$. The characterization is given in terms of a Muckenhoupt matrix $A_2$ condition on an associated Zibulski-Zeevi type matrix.
\end{abstract}

\maketitle

\bigskip

\section{Introduction}
For a fixed function $g\in L^2(\bR)$, the corresponding  Gabor system is the collection of functions
$$G(a,b,g)=\{e^{2\pi i m a x}g(x-bn) : m,n\in \Z\}.$$
Such systems were introduced by Gabor with the aim to create sparse and efficient time-frequency localized expansions of signals using  elements from $G(a,b,g)$.
A major contribution to the theory of Gabor systems was made by Daubechies et al.\ \cite{MR836025} by  studying the problem of obtaining expansions relative to $G(a,b,g)$ in a Hilbert space frame setup giving a much more systematic approach to such expansions. The frame approach in \cite{MR836025} paved the way for a very extensive study of the frame properties of Gabor systems, see e.g.\ \cite{MR2428338,MR2744776,MR1955929} and references therein.
 
In this paper we consider  multiple-generated Gabor systems
of the form
$$G(1,p/q,\A)=\{e^{2\pi i m x}g(x-np/q) : m,n\in \Z, g\in \A\},$$
with $\A$ be a finite subset of $L^2(\R)$ and $p,q\in\bN$.

The are a number of interesting stability questions related to $G(1,p/q,\A)$. One immediate question is about completeness. 
We let $\G=\overline{\text{Span}}(G(1,p/q,\A))$. The question is then whether $\G=L^2(\bR)$ or $\G$ is a proper subspace of $L^2(\bR)$. This question was addressed by 
Zibulski and Zeevi in \cite{MR1448221}, where a complete characterization of completeness is given in terms of properties of a certain corresponding matrix generated by the so-called Zak transform. The Zak transform based approach by Zibulski and Zeevi will play a central role in the present paper.
 
 Frame properties of $G(1,p/q,\A)$ were considered in \cite{MR1448221} and further studied in \cite{MR2346803,MR1460623,MR2077516}. Multiple-generated Gabor Riesz bases for $L^2(\R)$ and/or for $\G$ (so-called Riesz sequences) we also characterized by Bownik and Christensen in \cite{MR2346803}, see also \cite{MR2077516} for the single window case.

Expansions relative to a Riesz basis converge unconditionally. In the present paper we move one step further to the borderline case where the convergence might be conditional and depend on the precise ordering of the system $G(1,p/q,\A)$.

We recall that an ordered family $B=\{x_n:n\in\bN\}$  of vectors in a
Hilbert space $H$ is a Schauder basis for $H$ if for every $x\in H$ there
exists a unique sequence $\{\alpha_n:=\alpha_n(x):n\in\bN\}$ of
scalars such that
\begin{equation}\label{eq:S}
  \lim_{N\rightarrow\infty}\sum_{n=1}^N \alpha_nx_n=x
  \end{equation}
in the norm topology of $H$. Clearly, any  Riesz basis for $H$ is also a Schauder basis for $H$. For Riesz bases the convergence in \eqref{eq:S} is unconditional. However, it is    well known that conditional Schauder bases exist.

We give a complete characterization of when the system $G(1,p/q,\A)$ with the proper ordering forms a Schauder basis for $\G$ and $L^2(\R)$.  The problem of characterizing Gabor Schauder bases in the case of one generator was first considered by Heil and Powell
\cite{MR2278667}. They obtained a complete characterization in terms a so-called Muckenhoupt $A_2$ condition on a certain multivariate weight obtained by the Zak transform. Our result will reproduce the result obtained in \cite{MR2278667} in the case of a singleton $\A$ and $p=q=1$. 
Our characterization will be given in terms of a certain matrix Muckenhoupt $A_2$ condition first introduced by the author in \cite{MR2737763}.
\section{Zak transform analysis of Gabor systems}
Let us introduce some notation that will be used throughout the paper.
For $p\in\N$, we define the domain 
\begin{equation}\label{eq:Tp}
  T_p:=[0,1)\times [0,1/p).
  \end{equation}
 We call a measurable (periodic) matrix function $W:T_p\ra \C^{N\times N}$ a matrix weight if  $W(x)$ is non-negative definite Hermitian matrix for a.e.\ $x$, and $W$ is in $L^1$, i.e., the matrix norm $\|W\|$ belongs to $L^1(T_p)$.  
 
 We define the matrix weighted $L^2$, denoted $L^2(T_p,W)$ to be the set of vector functions $\vf:T_p\rightarrow \bC^N$ such that 
\begin{align*}\|\vf\|_{L^2(T_p,W)}^2&:=\int_0^1\int_0^{1/p} |W^{1/2}(x,u)\vf(x,u)|^2\, dx\,du\\&=\int_0^1\int_0^{1/p} \langle W(x,u)\vf(x,u),\vf(x,u)\rangle_{\C^N}\, dx\,du<\infty,
\end{align*}
where the Lebesgue measure is used.
We can turn $L^2(T_p,W)$ into a Hilbert space by factorizing over $N=\{\vf:\|\vf\|_{L^2(T_p,W)}=0\}$. Whenever $W(x)$ is strictly positive definite, $N$ is exactly the set of vector functions $\vf$ defined on $T_p$ that vanish a.e.\

The main tool we will use to analyze Gabor systems is the Zak transform. 
The Zak transform is defined for $f\in L^2(\R)$ by
$$Zf(x,u)=\sum_{k\in \Z} f(x-k) e^{2\pi i ku},\qquad x,u\in\T.$$
The Zak transform is a unitary transform of $f\in L^2(\R)$ onto $L^2(\T^2)$.

We now turn to the multiple-generated Gabor setup. Suppose we have $L$ generators $\A=\{g^\ell\}_{l=0}^{L-1}\subset L^2(\R)$. For fixed $p,q\in\N$, the associated Gabor system is given by
$$\G:=G(1,p/q,\A)=\{g_{m,n}^\ell\}_{m,n\in\Z},$$
with $g_{m,n}^\ell:=e^{2\pi i m x}g^\ell(x-np/q)$. A straightforward calculation shows (see e.g.\  \cite[Lemma 3.2]{MR2077516}),
\begin{equation}\label{eq:Za}
  Zg_{m,(nq+r)}^\ell(x,u)=e^{2\pi i m  x}e^{-2\pi i n p u} Zg^\ell\bigg(x-r\frac{p}{q},u\bigg),\qquad 0\leq r<q,
  \end{equation}
  For notational convenience, we put \begin{equation}\label{eq:emn}
  E_{m,n}(x,u):=e^{2\pi i mx}e^{-2\pi i n p u}.
  \end{equation}
Let us consider a finite expansion,
\begin{equation}\label{eq:f}
  f=\sum_{\ell=0}^{L-1}\sum_{r=0}^{q-1}\sum_{m,n\in\Z}c^\ell_{m,nq+r}g^\ell_{m,nq+r}.
  \end{equation}
Put
\begin{equation}\label{eq:tau}
  \tau_r^l(x,u)=\sum_{m,n\in\Z} c^\ell_{m,nq+r}E_{m,n}(x,u).
  \end{equation}
  Then by \eqref{eq:Za},
  \begin{equation}\label{eq:Z}
  Zf(x,u)=\sum_{\ell=0}^{L-1}\sum_{r=0}^{q-1} \tau_r^\ell(x,u)Zg^\ell\bigg(x-r\frac p q,u\bigg).
  \end{equation}
 We now follow  the approach of Zibulski and Zeevi \cite{MR1448221} and introduce a convenient matrix notation. We let $G^\ell:=G^\ell(x,u)$ be the $q\times p$-matrix given by
 \begin{equation}\label{eq:G}
  G^\ell_{r,k}=Zg^\ell\bigg(x-r\frac pq,u+\frac kp\bigg),\quad 0\leq r<q,\quad0\leq k<p.
  \end{equation} 
  We form the $Lq\times p$-matrix
\begin{equation}\label{eq:G1}
    G=\begin{bmatrix}G^0\\G^1\\\vdots\\G^{L-1}\end{bmatrix},
  \end{equation}
  and put 
  \begin{equation}\label{eq:We}
  W=GG^*\geq 0. 
  \end{equation}
  
  We mention that to study completeness in $L^2(\R)$ of the 
  system  $G(1,p/q,\A)$  and to study its frame properties one often turns to the  $p\times p$-matrix $G^*G$, see e.g. \cite{MR1448221,MR2346803}. For example, $\G=L^2(\R)$ if and only if $G^*G$ has full rank a.e. However, as we will se below, basis properties of 
  $G(1,p/q,\A)$ are more closely related to properties of the matrix in \eqref{eq:We}.
  
  Notice that $W$ given by \eqref{eq:We} is an $Lq\times Lq$-matrix, with entry $(sq+r,tq+\ell)$ given by
  $$\sum_{k=0}^{p-1} 
  Zg^s\bigg(x-r\frac p q,u+\frac k p\bigg)
  \overline{Zg^t\bigg(x-\ell\frac p q,u+\frac k p\bigg)}.$$

  Also notice that each entry in $W$ is in $L^1(T_p)$. This follows from the Cauchy-Schwarz inequality using that $Zg^s\in L^2(\bT^2)$, $s=0,\ldots L-1$.
  
  We now form the vector
  $\vt(u,x)\in \C^{Lq}$ by letting $(\tau(u,x))_{sq+r}=\tau_r^s(x,u)$, $0\leq s\leq L-1$, $0\leq r<	q$. We have the following result

  \begin{theorem}\label{th:a} Let $p,q\in\N$. Suppose $\A=\{g^0,\ldots, g^{L-1}\} \subset L^2(\R)$, $\G=\overline{\Span}\{G(1,p/q,\A)\}$, and $W$ is the non-negative definite matrix given by \eqref{eq:We}.   Then the map $\ZZ:L^2(T_p,W)\ra \G$ given by
\begin{equation}\label{eq:ZZ}
  \ZZ(\vf)=Z^{-1}\left(\sum_{\ell=0}^{L-1}\sum_{r=0}^{q-1} (\vf)_{\ell q+r}Zg^\ell\bigg(x-r\frac p q,u\bigg)\right),\quad \vf=[f_0,\ldots,f_{Lq-1}]^T,
  \end{equation}
is an  isometric isomorphism.
\end{theorem}

\begin{proof}
Let $\{\ve_j\}_{j=0}^{Lq-1}$ denote the standard basis for $\bC^{Lq}$. It follow from \eqref{eq:ZZ} and \eqref{eq:Za} that 
\begin{equation}\label{eq:ZZZ}
  \ZZ(\ve_{\ell q+r} E_{m,n})=g^\ell_{m,nq+r}.
  \end{equation}
Now take a vector-function
  $\vt(x,u)\in L^2(T_p,W)$ 
  of the form $$\vt(x,u):=\sum_{\ell=0}^{L-1}\sum_{r=0}^{q-1}\tau_r^\ell(x,u)\ve_{\ell q+r},$$
where each
\begin{equation*}\label{eq:tau1}
  \tau_r^\ell(x,u)=\sum_{m,n\in\Z} c^\ell_{m,nq+r}E_{m,n}(x,u)
  \end{equation*}
  is a trigonometric polynomial. We notice that by \eqref{eq:ZZZ} and linearity, 
  $$f:=\ZZ(\vt)=\sum_{\ell=0}^{L-1}\sum_{r=0}^{q-1}
  \sum_{m,n\in\Z}c^\ell_{m,nq+r}g^\ell_{m,nq+r}.$$
Hence, using $\langle f,f\rangle=
\langle Zf,Zf\rangle$ and \eqref{eq:Z} we obtain,
  \begin{align*}
  \langle f,f\rangle&=
\begin{aligned}[t]
  \int_0^1\int_0^1&
  \sum_{\ell=0}^{L-1}\sum_{r=0}^{q-1} \tau_r^\ell(x,u)Zg^\ell\bigg(x-r\frac p q,u\bigg)\\
  &\times\overline{\sum_{t=0}^{L-1}\sum_{s=0}^{q-1}\tau_s^t(x,u)Zg^t\bigg(x-s\frac p q,u\bigg)}\,dxdu\nonumber
  \end{aligned}\\
  &=
\begin{aligned}[t]
\int_0^1\int_0^{1/p}&
  \sum_{r,\ell}\sum_{s,t}\tau_r^\ell(x,u)\\
&  \times\bigg[\sum_{k=0}^{p-1} Zg^\ell\bigg(x-r\frac p q,u+\frac k p\bigg)
  \overline{Zg^t\bigg(x-s\frac p q,u+\frac k p\bigg)}\bigg]\overline{\tau_s^t(x,u)}\,dx du\label{eq:ff} \end{aligned} \\
  &=\|\vt\|_{L^2(T_p,W)}^2.
      \end{align*}
      The vectors with trigonometric polynomial entries  are dense in $L^2(T_p,W)$, and the image under $\ZZ$ of such vectors are clearly dense in $\G$. Hence, we may conclude that $\ZZ$ extends to an isometric isometry from $L^2(T_p,W)$ onto $\G$.
\end{proof}  

\section{Bi-orthogonal systems and Schauder bases}
Let us recall some elementary facts about Schauder bases and bi-orthogonal sequences in a Hilbert space $H$
Suppose  $B=\{x_n:n\in\bN\}$ is a Schauder basis for $H$, i.e.\ 
that for every $x\in H$ there
exists a unique sequence $\{\alpha_n:=\alpha_n(x):n\in\bN\}$ of
scalars such that
\begin{equation*}
  \lim_{N\rightarrow\infty}\sum_{n=1}^N \alpha_nx_n=x
  \end{equation*}
in the norm topology of $H$.  The unique choice of scalars, and the fact that we are in a Hilbert space, implies
that $x\rightarrow \alpha_n(x)$ is a continuous linear functional for every
$n\in\bN$.  Furthermore, for every $n\in\bN$, there
exists a unique vector $y_n$ such that $\alpha_n(x)=\langle
x,y_n\rangle$. It follows that
\begin{equation}
  \label{eq:4}
  \langle x_m,y_n\rangle=\delta_{m,n},\quad m,n\in\bN.
\end{equation}

A pair of sequences $(\{u_n\}_{n\in\bN},\{v_n\}_{n\in\bN})$ in $H$
is a {\em bi-orthogonal system} if $\langle
u_m,v_n\rangle=\delta_{m,n}$, $m,n\in\bN$. We say that
$\{v_n\}_{n\in\bN}$ is a {\em dual sequence} to $\{u_n\}_{n\in\bN}$,
and vice versa. 

A dual sequence is not necessarily uniquely
defined. In fact, it is unique if and only if the original sequence is
complete in $H$ (i.e., if the span of the original sequence is dense
in $H$). 

Suppose  $B=\{x_n:n\in\bN\}$ is complete, and has a unique dual sequence
$\{y_n\}$. Then $B$ is a Schauder basis for $H$ if and only if the partial sum
operators $S_N(x)= \sum_{n=1}^N\langle x,y_n\rangle x_n$ are uniformly
bounded on $H$. 

Finally, we call $B=\{x_n:n\in\bN\}$ a basis sequence if it is a Schauder basis for its closed linear span.
\subsection{Bi-orthogonal Gabor systems}
We have the following result for multiple-generated Gabor systems. We let $\{\ve_j\}_{j=0}^{Lq-1}$ denote the standard basis for $\bC^{Lq}$.

\begin{prop}\label{th:1} Let $p,q\in\N$. Suppose $\A=\{g^0,\ldots, g^{L-1}\} \subset L^2(\R)$, and define  $\G=\overline{\Span}\{G(1,p/q,\A)\}$. Let $W$ be the non-negative definite matrix given by \eqref{eq:We}.   Then 
 $G(1,p/q,\A)$ has a uniquely defined bi-orthogonal system if and only if $W^{-1}\in L^1(T_p)$. In case  $W^{-1}\in L^1(T_p)$, the dual element to $g_{m,n q+r}^\ell$, $m,n\in \Z, 0\leq r<q$, is given by
$\widetilde{g_{m,nq+r}^\ell}:={p}\ZZ(W^{-1} E_{m,n}\ve_{\ell q+r})$.
\end{prop}
\begin{proof}
Let us first suppose that  $W^{-1}\in L^1(T_p)$. We  notice that 
\begin{align*}\|W^{-1}E_{m,n}\ve_{\ell q+r}\|_{L^2(T_p,W)}^2&=
\int_0^1\int_0^{1/p} \langle WW^{-1}E_{m,n}\ve_{\ell q+r},W^{-1}E_{m,n}\ve_{\ell q+r}\rangle_{\C^{Lq}}\, dx\,du\\
&=\int_0^1\int_0^{1/p} (W^{-1})_{\ell q+r,\ell q+r}(x,u) dxdu<\infty,
\end{align*}
so $\widetilde{g_{m,nq+r}^\ell}=\ZZ(W^{-1} E_{m,n}\ve_{\ell q+r})$ is well-defined. We have, for $n,n',m,m'\in \Z$, $0\leq r,r'<q$, and $0\leq \ell,\ell<L$,
\begin{align*}
\langle \widetilde{g_{m',n'q+r'}^{\ell'}} ,g_{m,nq+r}^\ell\rangle_{L^2(\R)}&=
p\langle  \ZZ(W^{-1} E_{m',n'}\ve_{\ell' q+r'}),\ZZ(\ve_{\ell q+r} E_{m,n})\rangle_{L^2(\R)}\\
&=p\langle  W^{-1} E_{m',n'}\ve_{\ell' q+r'},\ve_{\ell q+r} E_{m,n}\rangle_{L^2(T_p,W)}\\
&=p\int_0^1\int_0^{1/p} E_{m-m',n'-n} \langle \ve_{\ell' q+r'},\ve_{\ell q+r}\rangle_{\C^{Lq}} dxdu\\
&=\delta_{m,m'}\delta_{n,n'}\delta_{\ell,\ell'}\delta_{r,r'},
\end{align*}
 so $\{\widetilde{g_{m,n}^\ell}\}$ is indeed  a dual sequence to $G(1,p/q,\A)$.
 
We turn to the converse statement.  Suppose that $\{\widetilde{g_{m,n}^\ell}\}\subset \G\subseteq L^2(\bR)$ is a dual sequence to  $G(1,p/q,\A)$. The map $\ZZ$ is onto $\G$, so we can write $\widetilde{g_{m,n}^\ell}=\ZZ(\vf_{m,n}^\ell)$ for some $\vf_{m,n}^\ell\in L^2(T_p,W)$. Then
\begin{align*}
\delta_{m,m'}\delta_{n,n'}\delta_{\ell,\ell'}\delta_{r,r'}&=\langle  g_{m,nq+r}^\ell,\widetilde{g_{m',n'q+r'}^{\ell'}}\rangle_{L^2(\R)}\\
&=\langle \ZZ(\ve_{\ell q+r} E_{m,n}), \ZZ(\vf_{m',n'}^{\ell'})\rangle_{L^2(\R)}\\
&=\int_0^1\int_0^{1/p} (\vf_{m',n'}^{\ell'})^H W\ve_{\ell q+r}E_{m,n}(x,u)\,dxdu.
\end{align*}
We have $(\vf_{m',n'}^{\ell'})^H W\ve_{\ell q+r}\in L^1(T_p)$ since $\vf_{m',n'}^{\ell'}\in L^2(T_p,W)$ and $W\in L^1(T_p)$.
Also,  $\{pE_{m,n} \}$ forms on orthonormal trigonometric basis for $L^2(T_p)$, and since the Fourier transform is injective on $L^1(T_p)$, we get that for a.a.\ $(x,u)\in T_p$,
$$ (\vf_{m',n'}^{\ell'}(x,u))^H W(x,u)=pE_{-m,-n}(x,u) \ve_{\ell q+r}^T.$$
Hence $W$ has full rank a.e.\ and we may solve for $\vf_{m',n'}$ to get
$$\vf_{m',n'}=pW^{-1}(E_{m,n}\ve_{\ell q+r}).$$
\end{proof}
\subsection{Rectangular partial sums}
As before consider  $\A=\{g^0,\ldots, g^{L-1}\} \subset L^2(\R)$, and let $\G=\overline{\Span}\{G(1,p/q,\A)\}$.
Let us suppose that  $W^{-1}\in L^1(T_p)$ so that $G(1,p/q,\A)$ has a dual system in $\G$. For any $f=\ZZ(\vt)\in \G$ ,and $N_1,N_2\in \N$, we consider  the rectangular partial sum operator $T_{N_1,N_2}:\G\rightarrow \G$ given by
$$T_{N_1,N_2} f=\sum_{\ell=0}^{L-1}\sum_{r=0}^{q-1}
  \sum_{|m|\leq N_1, |n|\leq N_2}\langle f,\widetilde{g^\ell_{m,nq+r}}\rangle g^\ell_{m,nq+r}.$$
  We would like to study the boundedness properties of $\{T_{N_1,N_2}\}_{N_1,N_2}$ on $\G$. We mention that it is necessary for $T_{N_1,N_2}$ to be uniformly bounded on $\G$ if $G(1,p/q,\A)$ forms a Schauder basis for $\G$ with an ordering ``compatible'' with rectangular partial sums. However, proving uniform boundedness of $\{T_{N_1,N_2}\}_{N_1,N_2}$ is not quite enough to conclude that $G(1,p/q,\A)$ forms a Schauder basis for $\G$. We study this detain in Section \ref{sec:s} below.
  
From Proposition \ref{th:1} we obtain that $\widetilde{g_{m,nq+r}^\ell}:={p}\ZZ(W^{-1} E_{m,n}\ve_{\ell q+r})$. Therefore 
\begin{align*}{S_{N_1,N_2}} \vt&:=\sum_{\ell=0}^{L-1}\sum_{r=0}^{q-1}
  \bigg(\sum_{|m|\leq N_1, |n|\leq N_2}\langle \ZZ(\vt),\ZZ(pE_{m,n}W^{-1}\ve_{\ell q +r})\rangle_{L^2(\R)} E_{m,n}\bigg)\ve_{\ell q +r}\\
  &=
 \sum_{\ell=0}^{L-1}\sum_{r=0}^{q-1}
  \bigg(\sum_{|m|\leq N_1, |n|\leq N_2}\langle \vt,pE_{m,n}W^{-1}\ve_{\ell q +r}\rangle_{L^2(T_p,W)} E_{m,n}\bigg)\ve_{\ell q +r}\\
&=  \sum_{\ell=0}^{L-1}\sum_{r=0}^{q-1}
  \bigg(\sum_{|m|\leq N_1, |n|\leq N_2}\langle \vt,pE_{m,n}\ve_{\ell q +r}\rangle_{L^2(T_p)} E_{m,n}\bigg)\ve_{\ell q +r}.
  \end{align*}
Also notice that $\ZZ({T_{N_1,N_2}} \vt)=pS_{N_1,N_2} f$ using \eqref{eq:Za}.
And since $\|f\|_{L^2(\R)}=\|\vt\|_{L^2(T_p,W)}$, we deduce that
$$p\|S_{N_1,N_2}\|_{\G\rightarrow \G}=\|{T_{N_1,N_2}}\|_{L^2(T_p,W)\rightarrow L^2(T_p,W)}.$$

We can thus study the boundedness of $\{S_{N_1,N_2}\}_{N_1,N_2}$ by studying properties of the trigonometric system in $L^2(T_p,W)$. The connection between convergence of Fourier series in a weighted $L^2$-space and the so-called Muckenhoupt condition on the weight was first made precise in the seminal paper\cite{MR0312139} by Hunt, Muckenhoupt, and Wheeden. In this paper, we need a Muckenhoupt condition in the matrix setting. Muckenhoupt matrix weights and their connection to convergence of Fourier series was studied by Treil and Volberg in \cite{TreVol:1997a,TreVol:1997b}.
  The following class of product Muckenhoupt weights was introduced  in \cite{MR2737763}.
  \begin{defn}
  Let $W$ be a $N\times N$ matrix weight on $T_p$, i.e., a (1,1/p)-periodic measurable function
defined on $T_p$  whose values are positive semi-definite $N\times N$ matrices.
We say that $W$ satisfies the
  Muckenhoupt product $A_2$-matrix-condition provided that
\begin{equation}\label{muc}
  \sup_{I\times J}\bigg\|\bigg(\frac1{|I\times J|}\int_I\int_J W(x,u) dxdu\bigg)^{1/2}
\bigg(\frac1{|I\times J|}\int_I\int_J W^{-1}(x,u) dxdu\bigg)^{1/2}\bigg\|<\infty,
\end{equation}
where the sup is over all rectangles $I\times J \subset T_p$. The collection of all such
weights is denoted $\mathbb{A}_2(T_p)$. 
\end{defn}

We notice that $W\in \mathbb{A}_2$ implies that $W,W^{-1}\in L^1(T_p)$. It is not difficult to prove (see \cite[Lemma 3.4]{MR2737763}) that $W\in \mathbb{A}_2$  implies that $ W(x,\cdot)$ and $W(\cdot,u)$ are uniformly in the corresponding univariate matrix $A_2$ class for a.e.\ $x$, respectively a.e.\ $u$. So for the $u$ variable, we have
\begin{equation}\label{muc1}
\esssup_{u\in [0,1/p)}  \sup_{I}\bigg\|\bigg(\frac1{|I|}\int_I W(x,u) dxdu\bigg)^{1/2}
\bigg(\frac1{|I|}\int_I W^{-1}(x,u) dx\bigg)^{1/2}\bigg\|<\infty,
\end{equation}
and similar for $W(\cdot,u)$. This fact will be used in the proof of Theorem \ref{th:1}.

We can now call on the following product version of the Muckenhoupt-Hunt-Wheeden Theorem proved in 
by the author in \cite{MR2737763}.

\begin{theorem}\label{th:mmm}
Let $W:T_p\rightarrow \bC^{Lq\times Lq}$ be a matrix weight with $W,W^{-1}\in L^{1}(T_p)$. Let $\{\ve_j\}_{j=0}^{Lq-1}$ denote the standard basis for $\bC^N$.
Then the rectangular
partial sum operators
$$S_{{N_1,N_2}}\vf(x,u):=\sum_{\ell=0}^{L-1}\sum_{r=0}^{q-1}\bigg(\sum_{m,n\in\bZ:|m|\leq N_1,|n|\leq N_2} \langle \vf,pE_{m,n}\ve_{\ell q+r}\rangle_{L^2(T_p)}E_{m,n}(x,u)\bigg)\ve_{\ell q+r},$$ 
are uniformly bounded on $L_2(T_p;W)$ if and only if\, $W\in \mathbb{A}_2$.
\end{theorem}
\begin{remark}
Theorem \ref{th:mmm} is stated for weights on the torus $\bT^2$ in  \cite{MR2737763}, but the proof in \cite{MR2737763} translates verbatim to the case of the domain $T_p$.
\end{remark}

We can now deduce the following result.

\begin{corollary}\label{cor:1}
Let $p,q\in\N$. Suppose $\A=\{g^0,\ldots, g^{L-1}\} \subset L^2(\R)$, and define  $\G=\overline{\Span}\{G(1,p/q,\A)\}$. Let $W$ be the non-negative definite matrix given by \eqref{eq:G}.   Then the partial sum operators $\{S_{{N_1,N_2}}\}_{Ñ_1,Ñ_2\in\N}$ are uniformly bounded on $\G$ if and only if $W\in \mathbb{A}_2(T_p)$.
\end{corollary}
\subsection{Schauder bases} \label{sec:s}
We now turn to the question of turning the system $G(1,p/q,\A)$ into a Schauder basis. Guided by Corollary \ref{cor:1} we need to find an enumeration of $G(1,p/q,\A)$ that respects the rectangular partial sums.

We follow Heil and Powell \cite{MR2278667} and  consider the following class of enumerations.

\begin{defn}\label{def:enu}
  Let $\Lambda$ be the set containing all
enumerations 
$\{(k_j,n_j)\}_{j=1}^\infty$ of $\bZ^2$ defined in the
following recursive manner.
\begin{enumerate}
\item The initial terms $(k_1,n_1)\ldots (k_{J_1},n_{J_1})$ are either
$$(0,0),(1,0),(-1,0),\ldots (A_1,0),(-A_1,0)$$
or
$$(0,0),(0,1),(0,-1),\ldots,(0,B_1),(0,-B_1),$$
 for some positive integers $A_1$ or $B_1$.
\item If $\{(k_j,n_j)\}_{j=1}^{J_k}$ has been constructed to be of the form
$\{-A_k,\ldots,A_k\}\times \{-B_k,\ldots,B_k\}$ for some non-negative integers $A_k$, $B_k$, then terms are added to either the top and bottom or the left and right sides to obtain either the rectangle 
$$\{-A_k,\ldots,A_k\}\times\{-(B_k+1),\ldots,B_k+1\}$$
or
$$\{-(A_k+1),\ldots,A_k+1\}\times\{-B_k,\ldots,B_k\}.$$
For example, terms would first be added to the left side ordered as
$$(-(A_k+1),0),(-(A_k+1),1),(-(A_k+1),-1),\ldots, (-(A_k+1),B_k),(-(A_k+1),-B_k),$$
and likewise for the right side. Top and bottom proceed analogously.
\end{enumerate}
\end{defn}
 
Given $\sigma\in\Lambda$, we lift $\sigma$ to an enumeration $\tilde{\sigma}$ of $\{0,1,\ldots,L-1\}\times\bZ^2$ defined as follows
\begin{equation}\label{def:sigma}
 (0,\sigma(1)),(1,\sigma(1)),\ldots,(Lq-1,\sigma(1)),(0,\sigma(2)),\ldots,
(Lq-1,\sigma(2)),(0,\sigma(3)),\ldots
\end{equation}

Let us now assume that $\A=\{g^0,g^1,\ldots,g^{L-1}\}\subset L_2(\bR)$ such that
the system $\G(1,p/q,\A)$ has a unique dual system $\{\widetilde{g_{m,n}^\ell}\}$ in $\G=\overline{\Span}\{G(1,p/q,\A)\}$. 

For $j\in\N$, we write  $\tilde{\sigma}(j)=(\ell_j,m_j,n_j)$. We let $G(\tilde{\sigma}(j)):=g_{m_j,n_j}^{\ell_j}$, $F(\tilde{\sigma}(j))=\widetilde{g_{m_j,n_j}^{\ell_j}}$,
and introduce 
the partial sum operators
$$\cT^\sigma_J f:=\sum_{j=1}^J \langle f,F(\tilde{\sigma}(j))\rangle G(\tilde{\sigma}(j)),\qquad f\in \G.$$
We also need to consider the associated partial sum operator in $L_2(\bT^2;W)$. Let $\{\ve_j\}_{j=0}^{Lq-1}$ be the standard basis for $\C^{Lq}$ Put $e(\ell,m,n):=E_{m,n}\ve_\ell$, and
$\tilde{e}(\ell,m,n):=\ZZ(pE_{m,n}W^{-1}\ve_\ell)$. Then
$$\cS^\sigma_J\vt:=\sum_{j=1}^J \langle \vt,\tilde{e}(\tilde{\sigma}(j)\rangle_{L_2(T_p,W)} e(\tilde{\sigma}(j)),\qquad \tau\in L_2(T_p,W),$$
satisfies $\ZZ(\cS^\sigma_J\tau)=\cT^\sigma_J f$ for $\ZZ(\vt)=f\in \G$.

It is now immediate from our general discussion of Schauder bases that the following conditions are equivalent:
\begin{enumerate}
\item[(i)] The system $G(1,p/q,\A)$  is a Schauder basis for $\G$ with the ordering induced by $\sigma\in\Lambda$
\item[(ii)] The partial sum operators $T^\sigma_J $ are uniformly bounded on $\G$.
\end{enumerate}
This leads to the following result, Theorem \ref{th:2}. The first part of the proof of Theorem \ref{th:2}  follows from the approach  outlined by the author in \cite[Corollary 3.4]{MR2737763}. We include here the details for the benefit of the reader.

\begin{theorem}\label{th:2} Let $p,q\in\N$. Suppose $\A=\{g^0,\ldots, g^{L-1}\} \subset L^2(\R)$, and define  $\G=\overline{\Span}\{G(1,p/q,\A)\}$. Let $W$ be the non-negative definite matrix given by \eqref{eq:G}.   
Then the following
  statements are equivalent
  \begin{enumerate}
 \item[(a)] $\sup_{\sigma\in\Lambda}\sup_{J}\|\cT_J^\sigma\|<\infty.$
\item[(b)] $W\in \mathbb{A}_2(T_p)$.
  \end{enumerate}
Moreover, at the critical density $Lq=p$, $W\in \mathbb{A}_2(T_p)$ implies that $G(1,p/q,\A)$ forms a Schauder basis for $L^2(\bR)$. 
\end{theorem}

\begin{remark}
For $L=p=q=1$, the condition in  Theorem \ref{th:2} reduces to the scalar condition $|Zg^0|^2\in A_2(\bT^2)$, which is exactly the condition derived by Powell and Heil in \cite{MR2278667}.
\end{remark}

We will need the following facts for the proof of Theorem \ref{th:2}. We let  $D_N$ denote the univariate Dirichlet kernel  given by
\begin{equation}
  \label{eq:diri}
D_0(t)=1,\qquad D_N(t)=\frac{\sin 2\pi(N+1/2)t}{\sin\pi t},\qquad N\geq 1,  
\end{equation}
and for $f\in L_2(\bT)$, $N\geq 0$,
\begin{equation}\label{eq:sn}
  S_N(f):=\sum_{k=-N}^N \hat{f}(k)e^{2\pi i k\cdot}=f* D_N:=\int_{\bT} f(t)D_N(\cdot-t)\,dt.
  \end{equation}
We lift $S_N$ to the vector setting by letting
$$S_N(\vt):=\sum_{j=0}^{Lq-1} S_N(\langle \vt,\ve_j\rangle)\ve_j.$$
It is known that 
\begin{equation}\label{eq:bd}
  \esssup_{u\in [0,1/p)} \sup_{N} \|S_N\|_{L^2(\bT,W(\cdot,u))\rightarrow L^2(\bT,W(\cdot,u))}<\infty
  \end{equation}
 for weights $W$ satisfying the univariate $A_2$ condition \eqref{muc1}, see \cite[Corollary 3.2]{MR2737763}. This is very closely related to the  fact that the Hilbert transform is bounded for such weights, see \cite{TreVol:1997a,TreVol:1997b}. We refer to \cite{MR2737763} for further details. The same result of course holds for 
 $W(x,\cdot)$.

\begin{proof}[Proof of Theorem \ref{th:2}]
It suffices to consider the operator $\cS^\sigma_J\tau$ on $L^2(T_p,W)$.
Given a
  rectangle $$R=\{-N_1,\ldots,N_1\}\times\{-N_2,\ldots,N_2\},\qquad
  N_1,N_2\in\bN_0,$$ we can use Definition \ref{def:enu} to construct
  an enumeration $\sigma\in \Lambda$ such that
  $\sigma(\{1,\ldots,J\})=R$ for some $J\in \bN$.  Then
  $\cS^{\sigma}_{N\cdot J}={S_{N_1,N_2}}$, and therefore
  $\sup_{N_1,N_2\geq 0}\|S_{N_1,N_2}\|<\infty$. Hence, $W\in
  \mathbb{A}_2$ by Corollary~\ref{cor:1}.
Conversely, we 
fix $f\in \G$, and  pick $\sigma\in \Lambda$. For any $J$ we let $N_J$ be the
largest integer $N_j\leq J$ for which $\cT_{N_j}^\sigma f=T_{L,K}f$, for
some integers $L,K$. Now, by Corollary \ref{cor:1},
$$\|\cT_J^\sigma f\|_{L_2(\R)}\leq \|T_{L,K}f\|_{L_2(\R)}+
 \|(\cT_J^\sigma-T_{L,K})f\|_{L_2(\R)}\leq
C\|f\|_{L_2(\R)}+ \|(\cT_J^\sigma-T_{L,K})f\|_{L_2(\R)}.
$$
Hence, it suffices to bound the norm of the term
\begin{equation}
  \label{eq:diff}
  (T_J^\sigma-T_{L,K})f=\sum_{j=N_J}^J \langle f,F(\tilde{\sigma}(j))\rangle G(\tilde{\sigma}(j)).
\end{equation}
According to Definition \ref{def:enu}, the sum \eqref{eq:diff}
contains terms that have been added to the top and bottom or left and
right side of an rectangle. The cases are treated in a similar fashion. For definiteness, assume that \eqref{eq:diff} adds terms to the top of the rectangle.

We study the equivalent problem in
  $L^2(T_p,W)$.  Pick $\tau$ with $\ZZ(\tau)=f$, so $\ZZ(\cS_J^\sigma\tau)=\cT_J^\sigma f$. 
  Notice that the ordering $\tilde{\sigma}$ given by \eqref{def:sigma} ensures that the sum $(\cS_J^\sigma-S_{L,K})\tau$ 
  can be rewritten 
  \begin{align}
    \label{eq:top1}
(\cS_J^\sigma&-S_{L,K})\vt=\nonumber\\
&\sum_{\ell=0}^{L-1}\sum_{r=0}^{q-1}
\sum_{m=-M}^M \langle \vt, pe^{2\pi i m x}e^{-2\pi i (K+1)pu}\ve_{\ell q+r}\rangle_{L^2(T_p)} e^{2\pi i m x}e^{-2\pi i (K+1)pu}\ve_{\ell q+r}+R,    
  \end{align}
  where the remainder $R$ is a sum of at most $2Lq-1$ terms of the type
  $\langle \vt,pE_{m,n}{\mathbf e}_\ell\rangle E_{m,n}{\mathbf e}_\ell$.
We   observe that, in general,
$$\|\langle \vt,pE_{m,n}{\mathbf e}_\ell\rangle E_{m,n}{\mathbf e}_\ell\|_{L_2(\bT^2;W)}\leq p\|W\|_{L_1(T_p)}\|W^{-1}\|_{L_1(T_p)}\|\vt\|_{L_2(T_p,W)},$$
which follows from H\"older's inequality. We can therefore use the triangle inequality to uniformly estimate the remainder $R$ in \eqref{eq:top1} in terms of $\|\vt\|_{L_2(\bT^2;W)}$.
Next, we notice that,
\begin{align*}
\bigg\|\sum_{\ell=0}^{L-1}\sum_{r=0}^{q-1}
\sum_{m=-M}^M \langle \vt,& pe^{2\pi i m x}e^{-2\pi i (K+1)pu}\ve_{\ell q+r}\rangle_{L^2(T_p)}e^{2\pi i m x}e^{-2\pi i (K+1)pu}\ve_{\ell q+r}\bigg\|_{L_2(T_p,W)}\\&=
 \bigg\|\sum_{\ell=0}^{L-1}\sum_{r=0}^{q-1}
\sum_{m=-M}^M p\langle \vt e^{2\pi i (K+1)pu}, e^{2\pi i m x}\ve_{\ell q+r}\rangle_{L^2(T_p)} e^{2\pi i m x}\ve_{\ell q+r}\bigg\|_{L_2(T_p,W)}.
\end{align*}
For notational convenience, we define the vector function 
$$\vf(x):=p\int_0^{1/p} \vt(x,u) e^{2\pi i (K+1)pu}\,du=S_0(\vt(x,\cdot) e^{2\pi i (K+1)p\cdot}),$$
where $S_0$ is the $0$-order partial sum operator.
We recall that $W(\cdot,u)$ and $W(x,\cdot)$ are uniformly matrix $A_2$-weights for a.e.\ $u$ and a.e.\ $x$, respectively. We now use that the boundedness estimate \eqref{eq:bd} to obtain that
\begin{align*}
 \bigg\|\sum_{\ell=0}^{L-1}\sum_{r=0}^{q-1}
\sum_{m=-M}^M& p\langle \vt e^{2\pi i (K+1)pu}, e^{2\pi i m x}\ve_{\ell q+r}\rangle_{L^2(T_p)} e^{2\pi i m x}\ve_{\ell q+r}\bigg\|_{L_2(T_p,W)}\\&=
\int_0^{1/p}\int_0^1 |W^{1/2}(x,u)D_M*\vf|^2 dxdu\\
&\leq C \int_0^{1/p}\int_0^1 |W^{1/2}(x,u)\vf|^2 dxdu\\
 &= C  \int_0^{1/p}\int_0^1  |W^{1/2}(x,u)S_0(\vt(x,\cdot) e^{2\pi i (K+1)p\cdot})|^2 dudx\\
 &\leq C'  \int_0^{1/p}\int_0^1 |W^{1/2}(x,u)\vt e^{2\pi i (K+1)pu}|^2 dudx\\
 &=C'\|\vt\|_{L_2(T_p,W)}^2,
\end{align*}
where we also used that $S_0$ is bounded uniformly on $L_2(\bT;W(x,\cdot))$ for a.e.\ $x$.
 Collecting the estimates, we conclude that
$$\|(T_J^\sigma-T_{L,K})f\|_{L_2(\bR)}=
\|(S_J^\sigma-S_{L,K})\vt\|_{L_2(T_p,W)}\leq C'\|\vt\|_{L_2(T_p,W)}=C'
\|f\|_{L_2(\bR)},$$
with $C'$ independent of $J$.

For the last part we notice that whenever the $Lq\times Lq$ matrix $W=GG^*$ is positive definite a.e.\ then $p=Lq=\text{rank}(GG^*)\leq \text{rank}(G)\leq p$ a.e.\ since $G$ is of size $Lq\times p$.  Hence, $\text{rank}(G)=p$ a.e. The fact that $\G=L^2(\bR)$  now follows from \cite[Theorem 2]{MR1448221}.
\end{proof}

\begin{remark}
Suppose that  $G(1,p/q,\A)$ ordered by $\sigma\in \Lambda$ forms a Schauder basis for $\G=\overline{\Span}\{G(1,p/q,\A)\}$ with $\A=\{g^0,\ldots, g^{L-1}\}$. Then clearly by Theorem \ref{th:2} the corresponding weight $W$ satisfies $W\in A_2(T_p)$.  
It is easy to check directly that for any subset $\A'\subset \A$, $G(1,p/q,\A')$ forms a Schauder basis for $\G'=\overline{\Span}\{G(1,p/q,\A')\}$. Let  $W'$ be the weight corresponding to $G(1,p/q,\A')$, where we notice that $W'$ is a submatrix of $W$. We deduce from Theorem \ref{th:2} that the submatrix $W'$ must also belong to  $A_2(T_p)$. 
  \end{remark}

\section{Example}
We conclude this paper by giving an example of a conditional multiple-generated Gabor Schauder basis for $L^2(\R)$. Let us consider the case $L=p$, $p\geq 2$, and $q=1$. Take univariate polynomials $P_0(x),\ldots P_{L-1}(x)$ and exponents $a_0,\ldots,a_{L-1}\in\bR$ satisfying 
$$-1<\deg(P_j)a_j<1,\qquad j=0,1,\ldots, L-1.$$
Then it is well known that $|P_i|^{a_i}$ satisfied the scalar $A_2$-condition, i.e.,
\begin{equation}\label{eq:aa2}
  \sup_I\left(\frac1{|I|}\int_I |P_j|^{a_j} dx\cdot \frac1{|I|}\int_I |P_j|^{-a_j} dx\right)<+\infty,
  \end{equation}
where the sup is over all intervals $I\subset [0,1)$, see \cite[Chap.\ 5]{MR1232192}.
We now put 
$$g^\ell(x):=\chi_{[0,1)}(x+\ell)|P_\ell(x+\ell)|^{a_\ell/2}.$$
It is easy to verify that $g^\ell\in L^2(\R)$. Notice that
$$Zg^\ell(x,u)=|P_\ell(x)|^{a_\ell/2}e^{2\pi i \ell u}.$$
We now form the matrix $W=GG^*$ defined by \eqref{eq:We}. Notice that entry $r,s$ of $W$ is given by 
\begin{align*}
\sum_{k=0}^{L-1} Zg^r\bigg(x,u+\frac{k}L\bigg)\overline{Zg^s\bigg(x,u+\frac{k}L\bigg)}
&=\sum_{k=0}^{L-1} |P_r(x)|^{a_r/2}|P_s(x)|^{a_s/2} e^{2\pi i r(u+k/L)}e^{-2\pi i s(u+k/L)}\\
&=|P_r(x)|^{a_r/2}|P_s(x)|^{a_s/2}e^{2\pi i (r-s) u} \sum_{k=0}^{L-1}e^{2\pi i (r-s)k/L}\\
&=L\delta_{r,s} |P_r(x)|^{a_r}.
\end{align*}
It is now straightforward to use \eqref{eq:aa2} to verify that the diagonal matrix $W\in \mathbb{A}_2$. Hence, for $\A=\{g^0,\ldots,g^{L-1}\}$, the system
$G(1,L,\A)$ forms a Schauder basis for $L^2(\bR)$ according to Theorem \ref{th:2} since $Lq=L=p$.

Moreover, we notice that by choosing appropriate polynomials $P_j$, we can easily obtain a matrix $G$ containing unbounded row vectors on $[0,1)\times [0,1/p)$ and/or  row vectors not bounded away from 0 in norm on $[0,1)\times [0,1/p)$. It thus follows from \cite[Theorem 2.2]{MR2346803} that the corresponding system $G(1,L,\A)$ cannot form an unconditional Riesz basis for $L^2(\bR)$. We thus obtain  an example of a conditional multiple-generated Gabor Schauder basis for $L^2(\R)$.
\bibliographystyle{abbrv}

\end{document}